\documentclass[12pt]{article}
\usepackage{graphicx, times}
\usepackage{amsfonts}
\usepackage{amsmath}
\usepackage{amsthm}
\usepackage{amssymb}
\usepackage{fancyhdr}
\usepackage{indentfirst}
\usepackage{titlesec}
\usepackage{hyperref}
\usepackage{abstract}
\usepackage{refcount}

\newcommand{\al}{\alpha}

\newcommand{\R}{\mathbb{R}}

\newcommand{\cS}{\mathcal{S}}

\newcommand{\la}{\lambda}

\newcommand{\La}{\Lambda}
\newcommand{\si}{\sigma}
\newcommand{\Si}{\Sigma}
\newcommand{\de}{\delta}

\newcommand{\ep}{\epsilon}
\newcommand{\pr}{\prime}
\newcommand{\Om}{\Omega}

\newcommand{\ti}{\tilde}

\newcommand{\mL}{\mathcal{L}}

\newcommand{\M}{\mathcal{M}}

\newcommand{\p}{\partial}

\newcommand{\RR}{\mathbb{R}}

\newcommand{\rom}[1]{\expandafter\romannumeral #1}
\newcommand{\Rom}[1]{\uppercase\expandafter{\romannumeral #1}}

\newcommand{\Ind}{\mathrm{Ind}}

\newcommand{\beq}{\begin{equation}}
\newcommand{\eeq}{\end{equation}}
\newcommand{\beqs}{\begin{eqnarray*}}
\newcommand{\eeqs}{\end{eqnarray*}}
\newcommand{\beqn}{\begin{eqnarray}}
\newcommand{\eeqn}{\end{eqnarray}}
\newcommand{\beqa}{\begin{array}}
\newcommand{\eeqa}{\end{array}}

\begin{document}

\newtheorem{theorem}{Theorem}[section]
\newtheorem{proposition}[theorem]{Proposition}
\newtheorem{corollary}[theorem]{Corollary}
\newtheorem{claim}{Claim}

\theoremstyle{definition}
\newtheorem{definition}[theorem]{Definition}

\theoremstyle{plain}
\newtheorem{lemma}[theorem]{Lemma}

\newtheorem{remark}[theorem]{Remark}

\theoremstyle{plain}

\numberwithin{equation}{section}

 \titleformat{\section}
   {\normalfont\bfseries\large}
   {\arabic{section}}
   {12pt}{}
 \titleformat{\subsection}
   {\normalfont\bfseries}
   {\arabic{section}.\arabic{subsection}}
   {11pt}{}

\title{Existence of minimal surfaces of arbitrarily large Morse index}
\author{Haozhao Li\footnote{The author is partially supported by NSFC grant No. 11131007.} \quad and \quad  Xin Zhou\footnote{The author is partially supported by NSF grant DMS-1406337.}}
\maketitle

\pdfbookmark[0]{}{beg}

\renewcommand{\abstractname}{}    
\renewcommand{\absnamepos}{empty} 
\begin{abstract}
\textbf{Abstract:} We show that in a closed  3-manifold with a generic metric of positive Ricci curvature, there are minimal surfaces of arbitrary large Morse index, which partially confirms a conjecture by F. Marques and A. Neves. We prove this by analyzing the lamination structure of the limit of minimal surfaces with bounded Morse index.
\end{abstract}

\section{Introduction}

In 1982, S. T. Yau conjectured that every closed 3-manifold admits infinitely many smooth, closed, immersed minimal surfaces \cite[Problem 88]{Y}. This was recently confirmed by Marques and Neves when the ambient manifold has positive Ricci curvature \cite{MN13}; they proved the existence of infinitely many smooth, closed, embedded minimal hypersurfaces in manifold $M^{n+1}$ of positive Ricci curvature when $2\leq n\leq 6$. They \cite{MN13, M14, N14} also conjectured that generically the Morse indices of these minimal hypersurfaces should grow linearly to infinity. However, it is even unknown that whether the Morse indices of these minimal surfaces are bounded. In this paper, we address this question in dimension 3 by showing that generically the Morse indices of these minimal surfaces must grow up to infinity.

\vspace{0.5em}
Denote $M^3$ by a 3-manifold.  A metric $g$ on $M$ is called {\em bumpy}, if any closed embedded minimal surface has no nontrivial Jacobi fields. In \cite{W91}, B. White showed that bumpy metrics are generic. Our main result is the following:

\begin{theorem}\label{main theorem}
Let $M^3$ be a closed 3-manifold with a bumpy metric $g$ of positive Ricci curvature. Then there exist embedded minimal surfaces of arbitrarily large Morse index in $(M^3, g)$.
\end{theorem}

The problem of finding minimal surfaces with fixed topology and arbitrarily large Morse index in 3-manifolds has attracted a lot of interests. This problem originated from a well-known question of Pitts and Rubinstein (c.f. \cite{PR87, CM00}). Motivated by this question, Hass-Norbury-Rubinstein \cite{HNR03} constructed a smooth metric on any 3-manifold which admits embedded genus zero minimal surfaces of arbitrarily high Morse index. Colding-Hingston \cite{CH03} and Colding-De Lellis \cite{CD05} extended the construction (by different methods), respectively, to find metrics on any 3-manifold which admits minimal surfaces with no Morse index bound of genus 1 and genus $\geq 2$. 
Compared to these results, our result gives a large family of metrics on 3-manifold which admit minimal surfaces of arbitrarily large Morse index, while assuming no genus bound.


\vspace{1em}
In fact, Theorem \ref{main theorem} is a direct corollary of the following finiteness theorem for minimal surfaces of bounded Morse index in a bumpy metric and Marques-Neves's result in \cite{MN13}. Let $M^3$ be a closed 3-manifold with a bumpy metric $g$ of positive Ricci curvature, and $\M$   the space of embedded, closed, minimal surfaces in $(M, g)$. By \cite{MN13}, $\M$ contains infinitely many elements. Given $N\gg 1$, we denote $\M_N$ by
\beq\label{M_N}
\M_N=\{\Si\in\M: \quad \Ind(\Si)\leq N\}.
\eeq
Then we have 
\begin{theorem}\label{main theorem 2} 
Let $M^3$ be a closed 3-manifold with a bumpy metric $g$ of positive Ricci curvature. 
Then there are only finitely many elements in $\M_N$.
\end{theorem}
\begin{remark}
The result is not true in general if we remove the assumption of positive Ricci curvature, as Colding and Minicozzi \cite{CM99} constructed an open set of metrics on any compact 3-manifolds which admit stable embedded minimal tori of arbitrarily large area\footnote{We would like to thank the referee for pointing out this example to us.}.
\end{remark}

Theorem \ref{main theorem 2} is proved by studying the limit of a sequence of minimal surfaces of bounded Morse index. Convergence theorems for minimal surfaces have been studied extensively in the past, especially in 3-dimension. Choi-Schoen \cite{CS85} proved the compactness of embedded minimal surfaces of fixed topological type in a closed 3-manifold with positive Ricci curvature. Anderson \cite{An85} and White \cite{W87} respectively proved similar convergence and compactness results for minimal surfaces and surfaces that are stationary for parametric elliptic functionals in 3-manifold with bounded genus and bounded area. For higher dimensional case, in a recent work \cite{S15}, B. Sharp showed a convergence result for minimal hypersurfaces with bounded Morse index and bounded volume. A key technical condition of these results is the area bound. Without this assumption, Colding and Minicozzi developed a whole theory \cite{CM04I, CM04II, CM04III, CM04, CM15} to study the convergence of minimal surfaces with fixed genus but no area bound. The key point in Colding-Minicozzi's theory is to study the lamination convergence of minimal surfaces. Our proof of Theorem \ref{main theorem 2} utilizes this idea, and the convergence result in the proof can be thought as an analogy of Colding-Minicozzi's theory for minimal surfaces with fixed Morse index but no area bound.  

\vspace{1em}
At the end, we want to mention Colding-Minicozzi's finiteness theorem \cite{CM00} and a corollary. 
\begin{theorem}
(Colding-Minicozzi \cite[Theorem 1.1]{CM00}) Let $M^3$ be a closed orientable 3-manifold with a bumpy metric $g$, and $V, e$ fixed. There exist at most finitely many closed embedded minimal surfaces in $M$ with genus $e$ and area at most $V$.
\end{theorem}
By Yang-Yau \cite{YY80} and Choi-Wang \cite{CW83}, in a Riemmannian 3-manifold $M$ with $Ric_g \geq \La>0$, any closed embedded minimal surface with fixed genus $e$ has area bounded from above by $\frac{16\pi(1+e)}{\La}$ (see also \cite[\S 7.5]{CM11}). Combining with Marques-Neves's result \cite{MN13}, we have:
\begin{corollary}
Let $M^3$ be a closed 3-manifold with a bumpy metric $g$ of positive Ricci curvature. Then there exist embedded minimal surfaces of arbitrarily large genus in $(M^3, g)$.
\end{corollary}
 
\vspace{1em}
We sketch the proof of Theorem \ref{main theorem 2} as follows. If Theorem \ref{main theorem 2} were not true, let $\{\Si_i\}$ be a sequence in $\M_N$ (\ref{M_N}). The Morse index bound implies that $\Si_i$ are locally stable away from finitely many points, and hence converge locally smoothly to a minimal lamination $\mL$ outside these points by Schoen's curvature estimates \cite{S83}. Assume that $\mL$ is  orientable and consider the leaves of $\mL$. 
If there is no isolated leaf, then we can construct a positive Jacobi field along an accumulating leaf $\La$, and hence show that $\La$ has only removable singularities; so its closure $\overline{\La}$ is a smooth, embedded, stable minimal surface, which contradicts the positive Ricci curvature condition. 
If $\mL$ has an isolated leaf $\La$ for which the multiplicity of the  convergence of $\Si_i$ to $\La$ is infinite or greater than one, then we can also construct a positive Jacobi field, and hence get a contradiction by using the same argument as above. In the case of multiplicity one convergence, we construct a nontrivial Jacobi vector field which extends across the singularities, hence contradicting the bumpy condition. One main technical difficulty here is the fact that $\Si_i$ might not a priori converge in the sense of varifold. For the case that $\mL$ is non-orientable, we can get a contradiction by combining the above arguments with the structure of non-orientable hypersurfaces (c.f. \cite[\S 3]{Z12}).\\

 {\bf Acknowledgements}: Part of this work was done while the first author was visiting MIT and he wishes to thank MIT for their generous hospitality. The second author wants to thank Toby Colding, Bill Minicozzi and Rick Schoen for helpful conversation on this work. Finally, we thank the referee for very useful comments.

\section{Preliminaries}
In this section, we recall some basic facts on minimal surfaces in 3-manifolds.
Let $M^3$ be a closed 3-manifold with a Riemannian metric $g$. A {\em minimal surface} $\Si^2\subset M^3$ is a critical point for the area functional, or equivalently a smooth surface with mean curvature equal to zero. $\Si$ is called {\em two-sided}, if $\Si$ has trivial normal bundle, i.e. there exists a unit normal vector field $\nu$. The Jacobi operator for the second variation of the area functional is given by (c.f. \cite[\S 1.8]{CM11})
\begin{equation}\label{Jacobi operator}
L=\Delta_{\Si}+|A|^2+Ric_g(\nu, \nu),
\end{equation}
where $\Delta_{\Si}$ is the Laplacian operator  on $\Si$ with the induced metric, and $A$ is the second fundamental form of $\Si$, and $Ric_g$ is the Ricci curvature of $(M, g)$. $\la\in \RR$ is an eigenvalue of $L$ if there exists a smooth function $\phi\in C^{\infty}(\Si)$ such that $L\phi=-\la\phi.$
The {\em Morse index} of $\Si$, denoted by $\Ind(\Si)$, is the number of negative eigenvalues of $L$ counted with  multiplicity. $\Si$ is called stable if the Morse index is zero. When $Ric_g>0$, there is no two-sided, stable, closed or complete, minimal surface. The case for closed minimal surface is trivial by plugging in constant test function in the second variation of area; while the case for complete minimal surface is given in the appendix, c.f. Theorem \ref{no stable complete minimal surface}. A Jacobi field $\phi\nu$ on $\Si$ is a solution of $L\phi=0.$\\

For stable minimal surfaces,  we recall the curvature estimates due to Schoen \cite{S83} (see also \cite[Corollary 2.11]{CM11}).

\begin{theorem}\label{lem:Schoen}
Given $k>0$, if $\Si^2\subset B_{r_0}(x)\subset M^3$ is an immersed stable minimal surface with trivial normal bundle, where the sectional curvature of $K_M$ of $M$ satisfies $|K_M|\leq k^2$, the radius $r_0<\rho_1(\pi/k, k)$ and $\partial \Si\subset \partial B_{r_0}(x)$, then for some $C=C(k)>0$ and all $0<\si\leq r_0$,
\beq\label{Schoen curvature estimates}
\sup_{B_{r_0-\si}(x)\cap \Si}|A|^2\leq \frac C{\si^2}.
\eeq
\end{theorem}

Now we introduce the definition of lamination.
\begin{definition}\label{definition of lamination}
(cf. Thurston \cite{T80}, Colding-Minicozzi \cite{CM00})
A {\em codimension one lamination} of $M^3$ is a collection $\mL$ of smooth, disjoint, connected surfaces, denoted by $\La$ (called {\em leaves}) such that $\cup_{\La\in\mL}\La$ is closed. Moreover for each $x\in M$ there exist an open neighborhood $U$ of $x$ and a local coordinate chart, $(U, \Phi)$, with $\Phi(U)\subset \RR^3$ such that in these coordinates the leaves in $\mL$ pass through the chart in slices of the form $(\RR^2\times \{t\})\cap \Phi(U).$
\end{definition}

In \cite{CM04}, Colding-Minicozzi discussed the compactness of laminations. Recall that a {\em minimal lamination} is a lamination whose leaves are minimal. A sequence of laminations is said to converge if the corresponding coordinate maps converge.

\begin{theorem}\label{convergence of lamination}
(\cite[Proposition B.1]{CM04}) Let $M^3$ be a fixed 3-manifold. If $\mL_i\subset B_{2R}(x)\subset M$ is a sequence of minimal laminations with uniformly bounded curvatures (where each leaf has boundary contained in $\p B_{2R}(x)$), then a subsequence of $\mL_i$ converges in the $C^{\al}$ topology for any $\al<1$ to a Lipschitz lamination $\mL$ in $B_R(x)$ with minimal leaves.

\end{theorem}
\begin{remark}
Here the ``Lipschitz" regularity means that the coordinate charts of $\mL$ are Lipschitz, while each leave is still smooth.
\end{remark}

We need the removable singularity theorem for stable minimal surfaces.

\begin{lemma}\label{removable singularity for stable surface}
Given a 3-manifold $M$ and a point $p\in M$, let $\Si\subset M\backslash\{p\}$ be a connected, properly embedded, minimal surface with trivial normal bundle. If $\Si$ is stable in a local geodesic ball $B_{r_0}(p)$ and $B_{r_0}(p)\cap (\bar \Si\backslash \Si)=\{p\}$, then $\Si$ has a removable singularity at $p$.
\end{lemma}
\begin{remark}
Lemma \ref{removable singularity for stable surface} is proved for $M=\RR^3$ by Gulliver-Lawson \cite[Theorem 3.3]{GL86} and Colding-Minicozzi \cite[Lemma A. 35]{CM15}. 
We present a proof based on ideas by Colding-Minicozzi \cite[Lemma A. 35]{CM15} and White \cite{W87} for completeness. 
Note that this result also follows by combining \cite[Lemma 1]{G76} and \cite[Remark A.2]{S15}.
\end{remark}
\begin{proof}
Consider the blow up $\Si_i=\la_i^{-1}\Si$ (by embedding $M$ into some $\R^N$) of $\Si$ at $p$ by an arbitraty sequence of numbers $\la_i$ decreasing to $0$. By the scaling invariant curvature estimates (\ref{Schoen curvature estimates}), $\Si_i$ converge locally smoothly to a complete, stable, minimal surface $\Si^{\pr}$ in $\R^3\backslash\{0\}$. The Bernstein theorem for stable minimal surfaces with point singularities in $\R^3$ (c.f. \cite[Corollary 3.16]{GL86}\cite[Lemma A. 34]{CM15}) implies that $\Si^{\pr}$ is a union of planes. By an argument of White \cite[page 251]{W87}, this implies that, for small enough $r$, the distance function $x\rightarrow dist_M(x, p)$ of $M$ when restricted on $\Si\cap B_r(p)$ has no critical points except for strict local minima, and hence (by Morse theory) each connected component of $\Si\cap B_r$ is either a disk with exactly one critical point, or a punctured disk with no critical point. Take a connected component $\ti{\Si}$ of $\Si\cap B_r$ that is a punctured disk, then the removable singularity theorem for minimal surface with finite Euler characteristic by Gulliver \cite{G76} and Choi-Schoen \cite[Proposition 1]{CS85} implies that $\ti{\Si}$ has a removable singularity at $p$. So Maximum Principle implies that $\Si\cap B_r$ has only one component, and $(\Si\cap B_r)\cup\{p\}$ is an embedded minimal disk.
\end{proof}




We also need the following removable singularity theorem for minimal laminations with a point singularity. 

\begin{theorem}\label{removable singularity for stable lamination}
Give a minimal lamination $\mL$ of a punctured ball $B_r(p)\setminus \{p\}$ in a 3-dimensional Riemannian manifold $(M^3, g)$, if the leaf is either stable when it is orientable, or has a stable orientable double cover, then $\mL$ extends to a minimal lamination of $B_r(p)$.
\end{theorem}
\begin{remark}
This theorem is a direct corollary of the combination of Schoen's curvature estimates (c.f. Theorem \ref{lem:Schoen}) and the Local Removable Singularity Theorem of Meeks-P\'erez-Ros \cite[Theorem 1.1]{MPR13}. In particular, Schoen's curvature estimates applied to each orientable leave $\La$ (or its orientable double cover when $\La$ is nonorientable) implies that $|A_\La|(x)d_M(x, p)\leq C$ for all $x\in \La\subset \mL$, and this is the only requirement for $\mL$ to apply \cite[Theorem 1.1]{MPR13}.
\end{remark}


\section{Proof of Theorem \ref{main theorem 2}}

In this section, we prove Theorem \ref{main theorem 2}. First, we recall an elementary fact (c.f. Fischer-Colbrie \cite[Proposition 1]{FC85} and Sharp \cite[Lemma 3.1]{S15}):
\begin{lemma}\label{lem:Ben}
Suppose $\Si^n\subset M^{n+1}$ is a closed, smooth, embedded, minimal hypersurface with Morse index $N$. Given any disjoint collection of $N+1$ open sets $\{U_i\}_{i=1}^{N+1}$, $U_i\subset M$,  
then $\Si$ must be stable in $U_i$ for some $1\leq i\leq N+1$.
\end{lemma}

The next result is an application  of Lemma \ref{lem:Ben}.

\begin{lemma}\label{lem:Ben2}
Given a closed, smooth, embedded, minimal surface $\Si^2\subset M^3$ with index at most $N$, then for any $r>0$ small enough, there exist at most $N$ disjoint balls $\{B_r(p_i)\}_{i=1}^N$ of $M^3$ such that $\Si^2$ is stable on any ball $B_r(x)$ in $M\backslash \cup_{j=1}^NB_r(p_i).$

\end{lemma}
\begin{proof}Suppose that the conclusion were not true. Given any disjoint collection of small balls $\{B_r(p_i)\}_{i=1}^N$ of $M^3$, we can find a ball $B_r(x_1)\subset M\backslash \cup_{j=1}^NB_r(p_i)$ such that $\Si$ is unstable in $B_r(x_1)$. By Lemma \ref{lem:Ben} there exists a ball, which we assume to be $B_r(p_1)$, such that $\Si$ is stable in $B_r(p_1)$. We now consider the $N$ balls
$$\cS_1:=\{B_r(p_2), \cdots, B_r(p_N), B_r(x_1)\}.$$
If the set $\cS_1$ satisfies the property of the lemma, then we stop. Otherwise,
 there exists a ball $B_r(x_2)\subset M\backslash \Big(\cup_{j=2}^NB_r(p_j)\cup B_r(x_1)\Big)$
such that $\Si$ is unstable on $B_r(x_2)$. By Lemma \ref{lem:Ben} there exists a ball, which cannot be $B_r(x_1)$ and so we assume to be $B_r(p_2)$, such that $\Si$ is stable in $B_r(p_2)$.  Then we consider the $N$ balls
$$\cS_2:=\{B_r(p_3), \cdots, B_r(p_N), B_r(x_1), B_r(x_2)\}.$$
If the set $\cS_2$ satisfies the property of the lemma, then we stop. Otherwise,
we repeat the above arguments.  We will either stop at $N$ balls satisfying the property of the lemma, or we get the disjoint collection of balls
$$\cS_N:=\{B_r(x_1), \cdots, B_r(x_N)\}.$$
Note that $\Si$ is unstable in any balls $B_r(x_i)$. By Lemma \ref{lem:Ben} again, $\Si$ will be stable in any ball $B_r(p)\subset M\backslash \cup_{j=1}^NB_r(x_i)$.
The lemma is proved.\\
\end{proof}

Now we are ready to prove Theorem \ref{main theorem 2}.
\begin{proof}[Proof of Theorem \ref{main theorem 2}] We divide the proof into the following steps.\\

\textbf{Step 1:} By contradiction argument, we assume that $(M^3, g)$ is a closed 3-manifold with a bumpy metric $g$, and $\M_N$ contains infinitely many elements. Let $\{\Si_i\}$ be a sequence in $\M_N$. Fix $r>0$. Then for each $\Si_i$,   by Lemma \ref{lem:Ben2} there exist at most $N$ balls $\{B_r(p_{i, j})\}_{j=1}^N$ of $M$, such that $\Si_i$ is stable in any ball $B_r(p)\subset M\backslash \cup_{j=1}^N B_r(p_{i, j})$. Up to a subsequence $\{p_{i, j}\}$ converge to $\{p_{\infty, j}\}$ as points, and hence by shrinking $r$, we have that $\Si_i$ is stable in any ball $B_r(p)$ in $M\backslash \cup_{j=1}^N B_{2r}(p_{\infty, j})$ for large $i$. Also by basic differential topology, $\Si_i$ must separate $B_r(p)$ when $B_r(p)$ is a topological ball, e.g. when $r$ is less than the injective radius of $(M, g)$, and hence $\Si_i\cap B_r(p)$ is orientable and two-sided. By Theorem \ref{lem:Schoen}, the curvature is uniformly bounded for any $p\in M$ such that $B_r(p)\subset M\backslash \cup_{j=1}^N B_{2r}(p_{\infty, j})$, i.e.
\beq\label{local curvature estimates}
\sup_{B_{r-\si}(p)\cap \Si_i}|A_{\Si_i}|^2\leq \frac C{\si^2}.
\eeq
Therefore, a subsequence of $\{\Si_i\}$ (clearly each $\Si_i$ is a minimal lamination) converges on such a ball $B_{r/2}(p)$ to a minimal lamination by Theorem \ref{convergence of lamination}. By passing to further subsequences (still using the same notation), $\{\Si_i\}$ converge to a minimal lamination on $M\backslash \cup_{j=1}^N B_{2r}(p_{\infty, j})$. Letting $r \rightarrow 0$, then a further subsequence of $\{\Si_i\}$ will converge to a minimal lamination of smooth, embedded, disjoint, minimal surfaces $\mL=\cup_{\La\in\mL}\La$ in $M\backslash\{p_{\infty, 1}, \cdots, p_{\infty, N}\}$. Therefore, the leaves of $\mL$ have at most $N$ singular points. Furthermore, the convergence of $\Si_i$ to each leaf in $\mL$ is locally smooth in $M\backslash\{p_{\infty, 1}, \cdots, p_{\infty, N}\}$ by the curvature estimates, and hence the leaves of $\mL$ satisfy similar local curvature estimates (\ref{local curvature estimates}) as well.

By the Maximum Principle for minimal surfaces \cite[Corollary 1.28]{CM11}, different sheets can touch at most on the singular sets. We say a leaf $\La$ is isolated if given any compact subset $K\subset \La$, there exists a tubular neighborhood $U$ of $K$ in $M$, such that the intersection of any other leaf with $U$ is empty. Given a leaf $\La\in\mL$, then either $\La$ is isolated, or by the curvature estimates (\ref{local curvature estimates}) and the Maximum Principle \cite[Corollary 1.28]{CM11}$, \La$ must be an accumulating leaf, i.e. there exists a sequence of leaves $\{\La_i\}$ in $\mL$, such that $\La_i$ converge to $\La$ locally smoothly (with multiplicity one by the definition of lamination).\\

\textbf{Step 2:}
Suppose that there is no isolated leaf. Then pick up an arbitrary accumulating leaf $\La$. 
Let us first assume $\La$ is orientable. Denote $\nu$ by the unit normal vector field along $\La$.

We claim that $\La$ is stable. Let $\La_i$ be the sequence of leafs in $\mL$ which converge to $\La$. The fact that $\La_i$ are disjoint from $\La$ implies that possibly up to a subsequence $\La_i$ converge to $\La$ locally smoothly from one side. Here convergence from one-side can be explained as follows: given any local ball $B_r(p)$ such that the intersection $\La\cap B_r(p)$ is diffeomorphic to a disk, and that $\La$ separates $B_r(p)$ into two connected components $U_1$ and $U_2$ with $\nu$ pointing into $U_1$, then for $i$ large all the intersections $\La_i\cap B_r(p)$ must lie in  $U_1$. For $i$ large enough we can choose domains $\Om_i\subset \La$ exhausting $\La$ so that each $\La_i$ can be written as a normal exponential graph over $\Om_i$ of a function $w_i$. Furthermore, we can assume that the graphical functions $w_i$ are positive, i.e. $w_i>0$, as $\La_i$ lies in one-side of $\La$. Denote $u_i=w_i/w_i(p)$ for some fixed $p\in \La$. Then following similar argument as in \cite[Lemma A.1]{CM00}, $u_i$ converge to a positive solution $u>0$ of the Jacobi equation $Lu=0$ on $\La$. By \cite[Lemma 1.36]{CM11}, the existence of a positive solution $Lu=0$ implies that $\La$ is stable. 

If $\La$ is non-orientable, we can consider the orientable double cover $\pi: \ti{\La}\rightarrow \La$. We claim that $\ti{\La}$ is a stable immersed minimal surface. In fact, we can consider the pullback of normal bundle $\pi^*\big(\nu(\La)\big)$, where $\nu(\La)$ is the normal bundle of $\La$ inside $(M, g)$. Now $\pi^*\big(\nu(\La)\big)$ forms a double cover of a neighborhood of $\La$, and we can pull back the metric $g$ and the lamination $\mL$ to $\pi^*\big(\nu(\La)\big)$. Inside the pullback lamination, the zero section, which is isometric to $\ti{\La}$, is an accumulating leave, and is therefore stable by the same argument as above.

Now we have proven that each leaf is either orientable and stable, or has a stable double cover. Therefore Theorem \ref{removable singularity for stable lamination} applied to balls near each singular point $p_{\infty, j}$ implies that the limit lamination $\mL$ has removable singularity, and each leave extends to a smooth, embedded, complete, minimal surface. Pick a leaf $\overline{\La}$ in the extended lamination; by Lemma \ref{extension of stability}, $\overline{\La}$ is stable when it is orientable, or has a stable double cover, contradicting the condition $Ric_g>0$ via Theorem \ref{no stable complete minimal surface}.
\begin{remark}
After completing this paper, we realized that \cite{MPR10} proved similar stability result for accumulating leaves using a different method by constructing local calibrations. 
\end{remark}
~\


\textbf{Step 3:} Suppose that there exists an isolated leaf $\La\in \mL$. Then by Step 1 there exists a sequence of minimal surfaces $\{\Si_{i}\}_{i=1}^{\infty}\subset \M_N$, so that given any compact subsets $K$ of $\La$, and a tubular neighborhood $U$ of $K$ (such that the intersection of other leaves with $U$ is empty), then $\Si_{i}\cap U$ converge locally smoothly to $\La $.  
In this step, we consider the case when the multiplicity of the convergence is greater than one, and we will postpone the case of multiplicity one convergence to the next step.

By similar argument as in Step 2, 
we can simply assume that $\La$ is orientable with a unit normal vector field. For $i$ large enough we can choose domains $\Om_i\subset \La$ exhausting $\La$, and tubular neighborhoods $U_i$ of $\Om_i$ with $U_i\cap \Si_i=\Om_i$, so that $\Si_i\cap U_i$ decomposes as a multi-valued normal exponential graph over $\Om_i$. By embeddedness and orientability of $\La$, these sheets are ordered by height. Using arguments similar to those in the proof of \cite[Theorem 1.1]{CM00}, we can construct a positive solution $u$ to the Jacobi equation $Lu=0$ on $\La$. Therefore $\La$ is stable. 

Now there are two possibilities. The simpler case is when $\La$ is properly embedded, and Lemma \ref{removable singularity for stable surface} implies that $\La$ has removable singularities, and the closure $\overline{\La}$ is a closed, smooth, embedded, orientable minimal surface, which is stable by Lemma \ref{extension of stability}, contradicting $Ric_g>0$. If $\La$ is not properly embedded, we can consider the set-theoretical closure $\overline{\La}$ of $\La$ in $M$, and such a closure forms a sub-lamination $\mL_{\La}$ of $\mL$
\footnote{\label{Secondfootnote}
Since $\mL$ is closed, $\overline{\La}$ is a subset of $\mL$ as subsets of $M$. Given any local coordinate chart $(U, \Phi)$ as in Definition \ref{definition of lamination}, and an arbitrary point $p\in(\overline{\La}\setminus\La)\cap U$, we claim that the leaf of $\mL\cap U$ passing through $p$, i.e. $\Phi^{-1}(t_p)$ where $t_p$ is the $t$-coordinate of $\Phi(p)$ in $\Phi(U)$, is contained in $\overline{\La}$, and this implies that $\overline{\La}$ has a lamination structure. To show the claim, let $p_i\in\La$ be a sequence of points converging to $p$. Denote $t_i$ by the $t$-coordinate of $\Phi(p_i)$ in $\Phi(U)$, then $t_i\rightarrow t_p$. Therefore $\Phi^{-1}(t_i)$ converge to $\Phi^{-1}(t_p)$ as sets. Hence $\Phi^{-1}(t_p)$ 
contains all limit points of $\La$, and $\Phi^{-1}(t_p)\subset\overline{\La}$.}. 
Besides $\La$, this sub-lamination contains only accumulating leaves by definition. Following arguments in Step 2, $\mL_{\La}$ satisfies the requirement of Theorem \ref{removable singularity for stable lamination}, and has removable singularities. Using arguments in Step 2 again, we can get a contradiction to $Ric_g>0$.\\

\textbf{Step 4:} Now we focus on the case when the multiplicity of the convergence of $\Si_i$ to each leaf is one. In this case, we are going to construct a Jacobi field along a leaf, say $\La$. We will show that $\La$ extends to a closed minimal surface with a nontrivial Jacobi field, which violates the bumpy condition. 

First we will show that $\La$ has removable singularities. 
In fact, near each singular point $p\in \overline{\La}\backslash \La$, there is a small radius $r_p>0$ such that $\La\cap B_{r_p}(p)\backslash \{p\}$ is stable. If this is not true, by similar arguments as in \cite[Proposition 1]{FC85} and \cite[\S 4, Claim 2]{S15}, we can construct infinitely many vector fields with pairwise disjoint compact supports in $B_r(p)\backslash\{p\}$ for some $r>0$, along which the second variation of $\La$ are negative. Since the convergence of $\Si_i$ to $\La$ is locally smooth on $B_r(p)\backslash\{p\}$, the second variation of $\Si_i$ along these vector fields will be negative for $i$ large enough, and this is a contradiction to the fact the $\Si_i$ has bounded Morse index.  The fact that $\La$ has removable singularity follows exactly the same as in Step 3.

We have two possibilities similarly as above. If $\La$ is not properly embedded, we can get a contradiction by looking at the accumulating leaves of the sub-lamination formed by the set-theoretical closure of $\La$. In particular, as in Step 3, $\overline{\La}\setminus\La$ is non-empty, and contains all accumulating leaves. Take an accumulating leaf $\La^{\pr}\subset\overline{\La}\setminus\La$, and consider its set-theoretical closure $\overline{\La^{\pr}}$, then this new sub-lamination $\mL_{\La^{\pr}}$ of $\mL$ contains all stable leaves, and we reduce to a case similar to Step 2.  


Therefore $\La$ is properly embedded, and $\overline{\La}$ is a closed, smooth, embedded, minimal surface. We will construct a nontrivial Jacobi field along $\overline{\La}$ by modifying the proof in \cite[Theorem 1.1 and Lemma A.1]{CM00}. In fact, the key missing point in our case is the lack of a priori uniform area bound for the sequence $\Si_i$. Theorefore we could not a priori get the varifold convergence of $\Si_i$, so we can not use the Allard Theorem \cite{Al72} to derive smooth convergence of $\Si_i$ to $\overline{\La}$ as in \cite{CM00}. Our strategy is as follows. In this part, we assume that $\La$ is orientable. By the locally smooth convergence with multiplicity one and the fact that $\La$ is isolated, for $i$ large enough we can choose domains $\Om_i\subset \La$ exhausting $\La$, and tubular neighborhoods $U_i$ of $\Om_i$ with $U_i\cap \La=\Om_i$, so that $\Si_i\cap U_i$ decomposes as a normal exponential graph over $\Om_i$ of a function $u_i$. See Figure \ref{fig1}. As in \cite[equation (7)]{Si87}, $Lu_i$ vanishes up to higher order terms which converge to zero as $i\rightarrow 0$. Here $L$ is the Jacobi operator of $\La$. Let $h_i=u_i/|u_i|_{L^2(\Om_i)}$, then the $h_i$'s satisfy uniform local $C^{2, \al}$ estimates by similar arguments using elliptic theory as \cite[Lemma A.1]{CM00}, and hence converge (up to a subsequence) to a function $h$ on $\La$ locally smoothly with $Lh=0$.

It remains to show that $h$ is nontrivial and that $h$ extends smoothly across those discrete points in $\overline{\La}\backslash\La$. To check these two facts, we will show that near each $p\in \overline{\La}\backslash \La$, the $|u_i|$'s and hence $|h_i|$'s are bounded by a uniform multiple of their supremum along $\partial B_{\ep}(p)\cap\La$ for some small $\ep>0$ (see the following (\ref{uniform interior estimates})). Recall the following facts in the proof of \cite[Theorem 1.1]{CM00}. Consider the cylinders $N_{\ep}$ (in exponential normal coordinates) over $B_{\ep}(p)\cap \overline{\La}$. If $\ep$ is small enough, there exists a foliation by minimal graphs $v_t$ of some small normal neighborhood of $\overline{\La}$ in $N_{\ep}$, so that
$$v_0(x)=0 \textrm{ for all }x\in B_{\ep}(p)\cap\overline{\La}, \textrm{ and } v_t(x)=t \textrm{ for all }x \in\partial B_{\ep}(p)\cap\La,$$
and $v_t$ satisfy uniform Harnack inequality, i.e. $t/C\leq v_t\leq Ct$ for some $C>0$. Note that we have
\begin{claim}
$\Si_i$ converge to the limit lamination $\mL$ in Hausdorff distance.
\end{claim}
\begin{proof}
Suppose not; 
then there exist $\de>0$, and a subsequence of $\Si_i$ (still denoting by $\Si_i$), such that the intersection of the support of each $\Si_i$ with $M\backslash U_{\de}(\cup_{\La\in\mL}\overline{\La})$ (the complement set of a $\de$-neighborhood of $\cup_{\La\in\mL}\overline{\La}$) is nonempty. However, we know that $\Si_i$ converge smoothly outside at most $N$ points (on the support of $\mL$) by Step 1. Also by the monotonicity formula \cite[\S 17]{Si83}, $\Si_i$ have a uniform area lower bound in $M\backslash U_{\de}(\cup_{\La\in\mL}\overline{\La})$. It means that $\Si_i$ must converge smoothly to a nonempty subset of the lamination $\mL$ in $M\backslash U_{\de}(\cup_{\La\in\mL}\overline{\La})$, hence a contradiction.
\end{proof}

\noindent Therefore by the Maximum Principle \cite[Corollary 1.28]{CM11} and the Hausdorff convergence, we can follow the argument in the last paragraph of \cite[page 119]{CM00} to show that (see Figure \ref{fig1})
\begin{equation}\label{uniform interior estimates}
\sup_{x\in B_{\ep/2}(p)\cap \Om_i}|u_i(x)|\leq C\sup_{x\in \partial B_{\ep}(p)\cap\La}|u_i(x)|.
\end{equation}
\begin{figure}[h]
\centering
\includegraphics[height=1.5in]{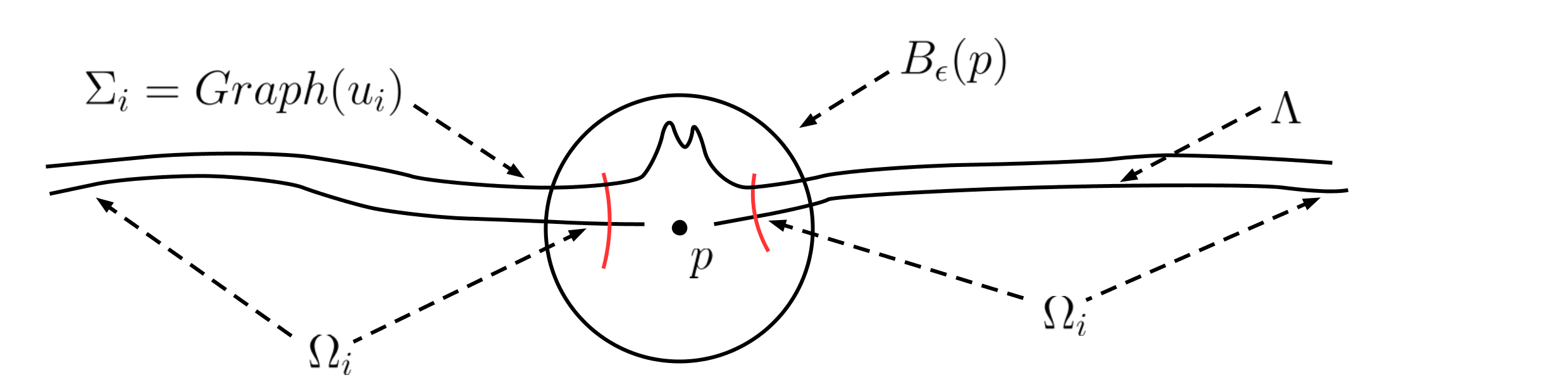}
\caption{$\Si_i$ near $p$}
\label{fig1}
\end{figure}
Now let us go back to check the behavior of $h_i$. If the limit of $h_i$ is zero, i.e. $h=0$, then $|h_i|$ converge to zero uniformly on any compact subsets of $\overline{\La}\backslash\La$ by the local smooth convergence. In particular $\sup_{\partial B_{\ep}(p)\cap\La}|h_i|$ will converge to zero uniformly, and by (\ref{uniform interior estimates}), $\sup_{B_{\ep/2}(p)\cap \Om_i}|h_i|$ will also converge uniformly to zero, hence contradicting the fact that $|h_i|_{L^2(\Om_i)}=1$. So $h$ is a nontrivial Jacobi field on $\La$. The same argument can also show that $\sup_{B_{\ep/2}(p)\cap \Om_i}|h_i|$ are uniformly bounded, and hence $h$ extends smoothly across $p$. So $h$ is a nontrivial Jacobi field on $\overline{\La}$, contradicting the bumpy condition, c.f. \cite[Theorem 2.2]{W91}.


If $\La$ is non-orientable, we can construct a nontrivial Jacobi field along $\overline{\La}$ similarly by lifting to a double cover. Since $\overline{\La}$ is closed and smooth, by the proof in \cite[Proposition 3.7]{Z12}, we can construct a double cover $\ti{M}$ of $M$, such that the pre-image of $\overline{\La}$ (which is also a double cover of $\overline{\La}$) is a smooth, embedded, orientable, minimal surface. Denote $\ti{\La}$ by the pre-image of $\La$ in $\ti{M}$. Let $\tau:\ti{M}\rightarrow\ti{M}$ by the inversion map such that $M=\ti{M}/\{id, \tau\}$ and $\overline{\La}=\overline{\ti{\La}}/\{id, \tau\}$. Let $\ti{\nu}$ be the unit normal vector field of $\ti{\La}$, then $\ti{\nu}$ is anti-symmetric, i.e. $\tau_*\ti{\nu}=-\ti{\nu}$. 
By counting the sheets of the pre-image $\ti{\Si}_i$ of $\Si_i$ in a small neighborhood of an arbitrary point $\ti{p}\in \ti{\La}$ in $\ti{M}$, which is isomorphic to a small neighborhood in $M$, we can show that 
$\ti{\Si}_i$ converge locally smoothly with multiplicity one to $\ti{\La}$. Then for $i$ large, $\ti{\Si}_i$ can be written as a normal exponential graph over $\ti{\La}$ of a function $u_i$ in the same manner as above.
\begin{claim}
$u_i$ is anti-symmetric with respect to $\tau$, i.e. $u_i\circ \tau=-u_i$.
\end{claim}
\begin{proof}
Fix an arbitrary $p\in\La$, and let $p_1, p_2$ be the pre-image of $p$ in $\ti{\La}$. For $i$ large, take $x_i\in \Si_i$ such that the nearest point projection of $x_i$ to $\La$ is $p$. Denoting $\ti{x}_i^1$ and $\ti{x}_i^2$ by the pre-image of $x_i$ in $\ti{M}$, then $\ti{x}_i^j=exp_{p_j}(u_i(p_j)\ti{\nu}(p_j))$, $j=1, 2$. As $\ti{\nu}$ is anti-symmetric and $\ti{x}_i^1, \ti{x}_i^2$ go to the same point in $M$, $u_i$ must be anti-symmetric.
\end{proof}
\noindent Theorefore, by working through the proof as above, the sequence of functions $h_i=u_i/|u_i|_{L^2}$ converge to a nontrivial solution $h$ of the Jacobi field equation on $\overline{\ti{\La}}$, and $h$ is anti-symmetric, i.e. $h\circ \tau=-h$. Then $\ti{X}=h\ti{\nu}$ is a symmetric Jacobi vector field on $\overline{\ti{\La}}$, and hence descends to a nontrivial Jacobi field on $\overline{\La}$, so contradicting the bumpy condition.

\end{proof}

\section{Appendix}

Here we show that stability can be extended across points on a minimal surface by logarithmic cutoff trick. We add a proof for completeness (see also \cite[Proposition 1.9]{GL86}). 
\begin{lemma}\label{extension of stability}
Let $\Si$ be a smooth minimal surface in $(M, g)$ with trivial normal bundle, and $p$ is an arbitrary point on $\Si$. If $\Si\backslash\{p\}$ is stable, then $\Si$ is stable.
\end{lemma}
\begin{proof}
Denote $r$ by the distance function to $p$ on $\Si$. Given $\ep>0$ small enough, and let $\eta(r)$ be a cutoff function defined by:
\begin{displaymath}
\left. \eta(r)= \Bigg\{ \begin{array}{ll}
0 \quad\quad \textrm{ if $r\leq \ep^2$},\\
\frac{2\log\ep-\log r}{\log \ep}, \quad \textrm{ if $\ep^2\leq r\leq \ep$},\\
1  \quad\quad \textrm{ if $r\geq \ep$},
\end{array}\right.
\end{displaymath}
Then $|\nabla \eta|\leq \frac{1}{|\log \ep|}\frac{|\nabla r|}{r}\leq \frac{1}{\ep^2|\log \ep|}$.
Now given $\varphi \in C^{\infty}(\Si)$, by the stability of $\Si\backslash\{p\}$,
\begin{equation}\label{stability inequality on punctured surface}
\int_{\Si}(|A|^2+Ric(\nu, \nu))\eta^2\varphi^2\leq \int_{\Si}|\nabla(\eta\varphi)|^2.
\end{equation}
The right hand side can be estimated as:
$$\int_{\Si}|\nabla(\eta\varphi)|^2\leq \int_{\Si}\eta^2|\nabla \varphi|^2+2\eta\varphi|\nabla \eta||\nabla\varphi|+\varphi^2|\nabla \eta|^2.$$
Therefore,
$$\int_{\Si}2\eta\varphi|\nabla \eta||\nabla\varphi|\leq C\int_{\Si\cap B_\ep(p)} \frac{1}{\ep^2|\log \ep|} \leq C^{\pr}\frac{1}{\ep^2|\log \ep|}\ep^2=\frac{C^{\pr}}{|\log \ep|}.$$
Also denoting $N=|\log \ep|$,
\begin{displaymath}
\begin{split}
\int_{\Si}\varphi^2|\nabla \eta|^2 &\leq \frac{C}{(\log \ep)^2}\int_{\Si\cap \big(B_\ep(p)\backslash B_{\ep^2}(p)\big)} \frac{1}{r^2} =\frac{C}{(\log \ep)^2}\sum_{l=1}^N\int_{\Si\cap \big(B_{e^{-2N+l}}(p)\backslash B_{e^{-2N+l-1}}(p)\big)}\frac{1}{r^2} \\
                                                    & \leq \frac{C^{\pr}}{(\log \ep)^2}\sum_{l=1}^N \frac{1}{(e^{-2N+l-1})^2} (e^{-2N+l})^2 =\frac{C^{\pr} e^2}{(\log \ep)^2}\cdot N=\frac{C^{\pr}e^2}{|\log \ep|}.
\end{split}
\end{displaymath}
Letting $\ep$ tend to zero in (\ref{stability inequality on punctured surface}), we can get the stability inequality for $\varphi$.
\end{proof}

Next we state a result which was essentially due to Fischer-Cobrie and Schoen \cite{FS80}, although it was not stated explicitly there. The proof is well-known to experts, and we include it for completeness. 
\begin{theorem}\label{no stable complete minimal surface}
Let $(M^3, g)$ be a three dimensional Riemannian manifold with positive Ricci curvature, then any complete immersed two-sided minimal surface $\Si$ is not stable.
\end{theorem}
\begin{proof}
If the result were not true, let $\Si$ be stable. We can assume that $\Si$ is non-compact. By \cite[Theorem 3]{FS80}, as $(M^3, g)$ has positive scalar curvature and non-negative Ricci curvature, $\Si$ endowed with the pull back metric is conformally equivalent to the complex plane $\mathbb{C}$. Since $\Si$ is stable, by \cite[Theorem 1]{FS80}, there exists a positive function $u$ on $\Si$, such that $Lu=\Delta_{\Si}u+\big(|A|^2+Ric_g(\nu, \nu)\big)u=0$, where $L$ is the Jacobi operator (\ref{Jacobi operator}) of $\Si$. Since $Ric(\nu, \nu)>0$, we have $\Delta_{\Si}u< 0$, i.e. $u$ is a superharmonic function. It is well known that any superharmonic function must be constant on $\mathbb{C}$ as $\mathbb{C}$ is parabolic (c.f. \cite[Proposition 1.37]{CM11}), hence $\Delta_{\Si}u=0$, so we get a contradiction.
\end{proof}



\vskip10pt
Haozhao Li,  Key Laboratory of Wu Wen-Tsun Mathematics, Chinese Academy of Sciences,  School of Mathematical Sciences, University of Science and Technology of China, No. 96 Jinzhai Road, Hefei, Anhui Province, 230026, China;  hzli@ustc.edu.cn.\\

Xin Zhou, Department of Mathematics, Massachusetts Institute of Technology, 77 Massachusetts Avenue, Cambridge, MA 02139, USA; xinzhou@math.mit.edu.

\end{document}